\documentclass[a4paper,11pt,leqno]{amsart}
\usepackage{amsmath, amsthm, amsfonts, amssymb, subfig, graphicx, mathrsfs,
}
\usepackage[mathscr]{eucal}
\usepackage{eucal}
\numberwithin{equation}{section}
\theoremstyle{plain}
\newtheorem{thm}{Theorem}[section]
\newtheorem{lem}[thm]{Lemma}
\newtheorem{prop}[thm]{Proposition}
\newtheorem{cor}[thm]{Corollary}
\newtheorem{rem}[thm]{Remark}

\theoremstyle{remark}

\theoremstyle{definition}

\def\B{\mathbb{B}}

\def\R{\mathbb{R}}
\def\N{\mathbb{N}}
\def\S{\Sigma}
\def\s{\sigma}

\def\e{\epsilon}
\def\g{\gamma}
\def\d{\delta}
\def\D{\Delta}

\def\G{\Gamma}
\def\a{\alpha}

\def\P{\mathscr{P}}
\newcommand{\dist}{\text{dist}}

\DeclareMathOperator{\diam}{diam}

\DeclareMathAlphabet{\mathscr}{OT1}{pzc}{m}{it}

\begin{document}

\title{Weak chord-arc curves and double-dome quasisymmetric spheres}
\date{\today}
\author{Vyron Vellis}
\address{Department of Mathematics and Statistics, P.O. Box 35 (MaD), FI-40014 University of Jyv\"askyl\"a, Jyv\"askyl\"a, Finland}
\email{vyron.v.vellis@jyu.fi}

\thanks{Research supported in part by the Academy of Finland project 257482.}

\begin{abstract}
Let $\Omega$ be a planar Jordan domain and $\a>0$. We consider double-dome-like surfaces $\S(\Omega,t^{\a})$ over $\overline{\Omega}$ where the height of the surface over any point $x\in\overline{\Omega}$ equals 
$\dist(x,\partial\Omega)^{\a}$. We identify the necessary and sufficient conditions in terms of $\Omega$ and $\a$ so that these surfaces are quasisymmetric to $\mathbb{S}^2$ and we show that $\S(\Omega,t^{\a})$ 
is quasisymmetric to the unit sphere $\mathbb{S}^2$ if and only if it is linearly locally connected and Ahlfors $2$-regular.
\end{abstract}

\maketitle
\date{\today}

\section{Introduction}\label{sec:intro}
A metric space which is quasisymmetric to the standard $n$-sphere $\mathbb{S}^n$ is called a quasisymmetric $n$-sphere. Quasisymmetric circles were completely characterized by Tukia and V\"ais\"al\"a
 in \cite{TuVa}. Bonk and Kleiner \cite{BonK} identified a necessary and sufficient condition for metric $2$-spheres to be quasisymmetric spheres. A consequence of their main theorem is that if a metric $2$-sphere is 
linearly locally connected (or LLC) and Ahlfors $2$-regular then it is a quasisymmetric $2$-sphere. Although the LLC property is necessary, there are examples of quasisymmetric $2$-spheres constructed by 
snowflaking procedures that fail the $2$-regularity \cite{bishop}, \cite{DToro}, \cite{Meyer}.

In this paper we consider a special case of the double dome-like surfaces constructed over planar Jordan domains $\Omega$ in \cite{VW2}. For a number $\a>0$ and a Jordan domain $\Omega\subset\R^2$ consider
 the $2$-dimensional surface in $\R^3$
\[ \S(\Omega,t^{\a}) = \{ (x,z) \colon x\in \overline{\Omega} \, , z = \pm (\dist(x,\partial\Omega))^{\a}\}.\]
As it turns out, for these surfaces the $2$-regularity is necessary for quasisymmetric parametrization by $\mathbb{S}^2$.

\begin{thm}\label{thm:2-regnec}
The surface $\S(\Omega,t^{\a})$ is quasisymmetric to $\mathbb{S}^2$ if and only if it is linearly locally connected and $2$-regular.
\end{thm}

What conditions on $\Omega$ and $\a$ ensure that $\S(\Omega,t^{\a})$ is a quasisymmetric $2$-sphere? When $\a>1$, Theorem \ref{thm:2-regnec} is trivial since, for any Jordan domain $\Omega$, the surface $\S(\Omega,t^{\a})$ 
is not linearly locally connected and hence not quasisymmetric to $\mathbb{S}^2$. If $\a=1$, $\S(\Omega,t)$ is quasisymmetric to $\mathbb{S}^2$ if and only if $\partial\Omega$ is a 
quasicircle \cite[Theorem 1.1]{VW2}. This result, combined with the fact that the projection of $\S(\Omega,t)$ on $\overline{\Omega}$ is a bi-Lipschitz mapping, and the fact that $\overline{\Omega}$ is $2$-regular if 
$\partial\Omega$ is a quasicircle, gives Theorem \ref{thm:2-regnec} when $\a=1$.

In the case $\a\in(0,1)$, the part of $\S(\Omega,t^{\a})$ near $\partial\Omega\times\{0\}$ resembles the product of $\partial\Omega$ with an interval. V\"ais\"al\"a \cite{Vais3} has shown that the 
product $\gamma \times I$ of a Jordan arc $\gamma$ and an interval $I$ is quasisymmetric embeddable in $\R^2$ if and only if $\gamma$ satisfies the chord-arc condition (\ref{eq:CA}). In view of this result, it is 
expected that the chord-arc property of $\partial\Omega$ is necessary for $\S(\Omega,t^{\a})$ to be a quasisymmetric sphere. Moreover, as the double-dome-like surface envelops the interior $\Omega$ above and below, the 
following condition on $\Omega$ is needed to ensure the linear local connectedness of the surfaces $\S(\Omega,t^{\a})$. We say that $\Omega$ has the \emph{level quasicircle property} (or LQC property) if there exist 
$\e_0>0$ and $K>1$ such that for all $\e\in[0,\e_0]$, the $\e$-\emph{level set} of $\Omega$
\[ \g_{\e} = \{x\in \overline{\Omega} \colon \dist(x,\partial\Omega) =\e\}\]
is a $K$-quasicircle. A consequence of Theorem 1.2 in \cite{VW2} is that \emph{if a planar Jordan domain $\Omega$ has the LQC property and $\partial\Omega$ is a chord-arc curve then $\S(\Omega,t^{\a})$ is 
quasisymmetric to $\mathbb{S}^2$ for all $\a\in(0,1)$}.

For these surfaces, the LQC property of $\Omega$ is essential: if a Jordan domain $\Omega$ does not have the LQC property then $\S(\Omega,t^{\a})$ is not quasisymmetric to $\mathbb{S}^2$ for any 
$\a\in (0,1)$; see Lemma \ref{lem:qcirclenec}. However, the chord-arc condition of $\partial\Omega$ is not necessary. Contrary to the intuition based on V\"ais\"al\"a's result \cite{Vais3}, we construct in Section 
\ref{sec:logn} a Jordan domain $\Omega$ whose boundary is a non-rectifiable curve and $\S(\Omega,t^{\a})$ is a quasisymmetric sphere for all $\a \in (0,1)$. 

Instead, only a weak form of the chord-arc condition is needed: 
a Jordan curve $\G$ is said to have the \emph{weak chord-arc condition} if there exists $N_0>1$ such that every subarc $\G'\subset \G$ with $\diam{\G'} \leq 1$ can be covered by at most $N_0(\diam{\G'})^{-1}$ subarcs 
$\G_i \subset \G'$ of diameter at most $(\diam{\G'})^2$. Under this terminology we have the following.

\begin{thm}\label{thm:main}
Let $\Omega$ be a Jordan domain and $\a\in (0,1)$. The following are equivalent.
\begin{enumerate}
\item The surface $\S(\Omega,t^{\a})$ is a quasisymmetric sphere.
\item The surface $\S(\Omega,t^{\a})$ is LLC and $2$-regular.
\item $\Omega$ has the LQC property and $\partial\Omega$ is a weak chord-arc curve.
\end{enumerate}
\end{thm}

An immediate consequence of Theorem \ref{thm:main} is that the existence of a quasisymmetric parametrization of $\S(\Omega,t^{\a})$ by $\mathbb{S}^2$ does not depend on $\a \in (0,1)$. In particular, if $\S(\Omega,t^{\a})$ 
is quasisymmetric to $\mathbb{S}^2$ for some $\a \in (0,1)$ then it is quasisymmetric to $\mathbb{S}^2$ for all $\a \in (0,1)$.

Although not neccessarily rectifiable, any weak chord-arc quasicircle has Hausdorff dimension equal to $1$. In fact, a slightly stronger result holds true.

\begin{thm}\label{thm:assouad}
If a quasicircle satisfies the weak chord-arc property then it has Assouad dimension equal to $1$.
\end{thm}

The example in Section \ref{sec:sqrtn}, however, shows that the converse is false. Moreover, since the Assouad dimension is larger than the Hausdorff, upper box counting and lower box-counting dimensions, the conclusion of 
Theorem \ref{thm:assouad} holds for any of these dimensions. A consequence of Theorem \ref{thm:assouad} is that if $\Omega$ is a Jordan domain and $\partial\Omega$ has Assouad dimension greater than $1$, then the surface 
$\S(\Omega,t^{\a})$ is not quasisymmetric to $\mathbb{S}^2$ for any $\a\in(0,1)$.

\medskip

In Section \ref{sec:weakca} we define an index that measures how much a curve $\G$ deviates from being a chord-arc curve on a certain scale, and we discuss the weak chord-arc condition. The proofs of Theorem \ref{thm:main} 
and Theorem \ref{thm:assouad} are given in Section \ref{sec:mainresults}. Finally, two examples, based on homogeneous snowflakes, illustrating the weak chord-arc condition are presented in Section \ref{sec:snowflakes}.

\medskip

\textbf{Acknowledgments:} The author is grateful to Pekka Pankka, Kai Rajala and his adviser Jang-Mei Wu for stimulating discussions and valuable comments on the manuscript.

\section{Preliminaries}\label{sec:prelim}

\subsection{Definitions and notation}
An embedding $f$ of a metric space $(X,d_X)$ into a metric space $(Y,d_Y)$ is said to be $\eta$-\emph{quasisymmetric} if there exists a homeomorphism $\eta \colon [0,\infty) \rightarrow [0,\infty)$ such 
that for all  $x,a,b \in X$ and $t>0$ with $d_X(x,a) \leq t d_X(x,b)$,
\[d_Y(f(x),f(a)) \leq \eta(t)d_Y(f(x),f(b)). \]

A metric $n$-sphere $\mathcal S$ that is quasisymmetrically homeomorphic to  $\mathbb S^n$ is called a \emph{quasisymmetric sphere} when $n\ge 2$, and a \emph{quasisymmetric circle} when $n=1$.

Beurling and Ahlfors \cite{BerAhl} showed that a planar Jordan curve is a quasisymmetric circle if and only if it is a $K$-quasicircle ($K\geq1$), i.e., the image of the unit circle $\mathbb S^1$ under a 
$K$-quasiconformal homeomorphism of $\mathbb{R}^2$. A geometric characterization due to Ahlfors \cite{Ah} states that a Jordan curve $\g$ is a $K$-quasicircle if and only if it satisfies the 
\emph{2-point condition}:
\begin{equation}\label{eq:3pts}
\text{there exists } C>1 \text{ such that for all }  x,y \in \g, \, \,\diam{\g(x,y)} \leq C|x-y|,
\end{equation}
where the distortion $K$ and the $2$-point constant $C$ are quantitatively related. Here and in what follows, given two points $x,y$ on a metric circle $\g$ we denote by $\g(x,y)$ the subarc of $\g$ 
connecting $x$ and $y$ of smaller diameter, or  either subarc when both have the same diameter. Tukia and V\"ais\"al\"a \cite{TuVa} proved that a metric circle is a quasisymmetric circle if and only if it is 
doubling and satisfies condition (\ref{eq:3pts}).

For the rest, a Jordan arc is any proper subarc of a Jordan curve and a quasiarc is any proper subarc of a quasicircle. 

A rectifiable Jordan curve $\g$ in $\R^2$ is called a \emph{chord-arc curve} if 
\begin{equation}\label{eq:CA}
\text{there exists } c>1 \text{ such that for all }  x,y \in \g, \, \, \ell(\g'(x,y)) \leq c|x-y|,
\end{equation}
where $\g'(x,y)$ is the shortest component of $\G\setminus \{x,y\}$ and $\ell(\cdot)$ denotes length. It is easy to see that a rectifiable curve $\g$ is a chord-arc curve if and only if $\ell(\g(x,y)) \leq C|x-y|$ for some $C>1$;
here constants $c$ and $C$ are quantitatively related .

The notion of linear local connectivity generalizes the $2$-point condition (\ref{eq:3pts}) on curves to general spaces. A metric space $X$ is $\lambda$-\emph{linearly locally connected} 
(or $\lambda-\text{LLC}$) for $\lambda\geq 1$ if the following two conditions are satisfied.
\begin{enumerate}
\item ($\lambda-\text{LLC}_1$) If $x\in X$, $r>0$ and $y_1,y_2 \in B(x,r)\cap X$, then there exists a continuum $E\subset B(x,\lambda r)\cap X$ containing $y_1,y_2$.
\item ($\lambda-\text{LLC}_2$) If $x\in X$, $r>0$ and $y_1,y_2 \in X \setminus B(x,r)$, then there exists a continuum $E\subset X \setminus B(x,r/\lambda)$ containing $y_1,y_2$.
\end{enumerate}

A metric space $X$ is said to be \emph{Ahlfors $Q$-regular} if there is a constant $C>1$ such that the $Q$-dimensional Hausdorff measure $\mathcal{H}^Q$ of every open ball $B(a,r)$ in $X$ satisfies
\begin{equation}\label{eq:2regdefn}
C^{-1}r^Q \leq \mathcal{H}^Q(B(a,r)) \leq Cr^Q,
\end{equation}
when $0<r\leq \diam{X}$.

Bonk and Kleiner found in \cite{BonK} an intrinsic characterization of quasisymmetric $2$-spheres and then derived a readily applicable sufficient condition.

\begin{thm}[{\cite[Theorem 1.1, Lemma 2.5]{BonK}}]\label{thm:BonkKleiner}
Let $X$ be an Ahlfors $2$-regular metric space homeomorphic to $\mathbb{S}^2$. Then $X$ is quasisymmetric to  $\mathbb{S}^2$ if and only if $X$ is LLC.
\end{thm}

For $x \in \R^n$ and $r>0$, define $B^n(x,r) = \{ y \in \mathbb{R}^n \colon |x-y| < r\}$. In addition, let $\R^3_+ = \{x=(x_1,x_2, x_3)  \in \R^3 \colon x_3 \geq 0 \}$ be the upper half-space of $\R^3$ and 
$\R^3_- = \{x  \in \R^3 \colon x_3 \leq 0 \}$ be the lower half-space. For any $a = (a_1,a_2,a_3)\in \R^3$, denote by 
\[\pi(a) = (a_1,a_2,0)\]
the projection of $a$ on the plane $\R^{2}\times\{0\}$. For $x,y \in \R^2$, 
denote by $[x,y]$ the line segment having $x,y$ as its end points.

In the following, we write $u\lesssim v$ (resp. $u \simeq v$) when the ratio $u/v$ is bounded above  (resp. bounded above and below) by positive constants. These constants may vary, but are described 
in each occurrence.

\subsection{Geometry of level sets}
For a planar Jordan domain $\Omega$ and some $\e>0$ define the $\e$-level set
\[ \g_{\e} = \{ x \in \Omega \colon \dist(x,\partial\Omega) = \e\}.\]
In general the sets $\g_\e$ need not be connected and if connected need not be curves see \cite[Figure 1]{VW}. We say that $\Omega$ has the \emph{level quasicircle property} (or LQC \emph{property}), 
if there exist $\e_0>0$ and $K\geq 1$ such that the level set $\g_{\e}$ is a $K$-quasicircle  for every $0\leq \e \leq \e_0$. Sufficient conditions for a domain $\Omega$ to satisfy the LQC property have
 been given in \cite{VW} in terms of the \emph{chordal flatness} of $\partial \Omega$, a scale-invariant parameter measuring the local deviation of subarcs from their chords.

We list some properties of the level sets from \cite{VW} which are used in the proof of Theorem \ref{thm:main}.

\begin{lem}[{\cite[Lemma 4.1]{VW}}]\label{lem:levelsubarc}
Suppose that $\Omega$ is a Jordan domain and for some $\e>0$ the set $\g_{\e}$ is a Jordan curve. If $\sigma$ is a closed subarc of $\partial\Omega$ then the set 
$\sigma' = \{z\in\Omega \colon \dist(z,\partial\Omega) = \dist(z,\sigma) = \e\}$ is either empty or a subarc of $\g_{\e}$.
\end{lem}

\begin{lem}[{\cite[Lemma 6.1]{VW}}]\label{lem:levelcalemma}
There is a universal constant $c_0 >1$ such that if $\sigma$ is a Jordan arc and $\e>3\diam{\s}$ then the set $\s' = \{x\in\R^2 \colon \dist(x,\sigma) = \e \}$ is a $c_0$-chord-arc curve.
\end{lem}

\begin{lem}[{\cite[Theorem 1.3]{VW}}]\label{lem:levelcathm}
Let $\Omega$ be a Jordan domain that has the LQC property. If $\partial\Omega$ is a $c$-chord-arc curve then there exist $\e_0 >0$ and $c'>1$ such that $\g_{\e}$ is a $c'$-chord-arc curve for all $\e\in[0,\e_0]$.
\end{lem}

\section{A weak-chord arc property}\label{sec:weakca}

\subsection{A chord-arc index}\label{sec:ca-index}
Let $\G$ be a Jordan curve or arc. A \emph{partition} $\P$ of $\G$ is a finite set of mutually disjoint (except for their endpoints) subarcs of $\G$ whose union is all of $\G$. We denote with $|\P|$
 the number of elements a partition $\P$ has. A \emph{$\d$-partition} of $\G$, for $\d \in (0,1)$, is a partition $\P$ of $\G$ such that for each $\G' \in \P$
\[ \frac{\d}{2}\diam{\G} \leq \diam{\G'} \leq \delta \diam{\G}.\]
Standard compactness arguments show that every Jordan curve or arc has $\d$-partitions for all $\d\in(0,1)$. Moreover, if $\P$ is a $\d$-partition of $\G$ then $|\P| \geq 1/\d$.

For a partition $\mathscr{P}$ of $\G$ define
\[ M(\G,\P) = \frac{1}{\diam{\G}}\sum_{\G'\in\P}\diam{\G'}.\]
The finite number $M(\G,\P)$ measures, on the scale of $\P$, the deviation of $\G$ from being a chord-arc curve. When $\ell(\G)$ is finite, the number $M(\G,\P)$ is an estimation of $\ell(\G)/\diam{\G}$:
 for any $\e>0$, there exists a partition $\P$ of $\G$ such that $|M(\G,\P)-\ell(\G)/\diam{\G}|<\e$. On the other hand, if $\ell(\G)$ is infinite then for all $M>0$ there exists a partition $\P$ of $\G$ 
such that $M(\G,\P) >M$. The properties of the gauge $M(\G,\P)$ are summarized in the following lemma.

\begin{lem}\label{lem:Mindex}
Suppose that $\G$ is a Jordan curve or arc.
\begin{enumerate}
\item If $\P$ is a $\d$-partition then $\frac{1}{2}|\P| \d \leq M(\G,\P) \leq |\P|\d$.
\item Suppose that $\P =\{\G_1,\dots,\G_N\}$ is a partition of $\G$ and for each $i=1,\dots,N$, $\P_i$ is a partition of $\G_i$. Then,
\[ M(\G,\bigcup_{i=1}^N\P_i) \geq M(\G,\P).\]
\item Assume that $\G$ is a $K$-quasiarc and $\d\in(0,1)$. There exists $M_1>1$ depending on $K$ such that $M(\G,\P)/M(\G,\P')\leq M_1$ for any two $\d$-partitions $\P,\P'$ of $\G$.
\end{enumerate}
\end{lem}

\begin{proof}
Property \emph{(1)} is an immediate consequence of the definition. For property \emph{(2)} simply note that $\diam{\G_i} \leq \sum_{\G'\in\P_i}\G'$. Finally, to show \emph{(3)}, let $\P = \{\G_1,\dots,\G_{N}\}$ and 
$\P' = \{\G_1',\dots,\G'_{N'}\}$ be two $\d$-partitions of $\G$. The $2$-point condition of $\G$ implies that there exists $N_0$ depending on $K$ such that each $\G_i$ can contain at most $N_0$ subarcs $\G_j'$. Thus, 
$N' \leq N_0 N$ and \emph{(3)} follows from \emph{(1)}.
\end{proof}

In the proof of Lemma \ref{lem:Mindex} we used a covering property of quasicircles which we prove in the following lemma.

\begin{lem}\label{lem:qcircles}
For each $K>1$ there exists a number $C>1$ depending only on $K$ such that, if $\G$ is a $K$-quasiarc and $\lambda\in(0,1]$ then, each partition $\{\G_1,\dots,\G_N\}$ with $\diam{\G_i} \geq \lambda \diam{\G}$ 
satisfies $N \leq C/\lambda^2$.
\end{lem}

\begin{proof}
We may assume that $\diam{\G}=1$. Fix $\lambda \in (0,1)$ and let $\{\G_1,\dots,\G_N\}$ be a partition of $\G$ with $\diam{\G_i} \geq \lambda$. Let $A=\{x_0,x_1,\dots,x_N\}$ be the set of 
endpoints of $\G_1,\dots,\G_N$. The $2$-point condition of $\G$ implies that there exists $C>1$ depending on $K$ such that $|x-y| \geq \lambda/C$ for all $x,y\in A$. Since all points of $A$ lie in $\R^2$ which is a doubling, 
space, there is a universal constant $C_0 >1$ such that $N \leq C_0 (\lambda/C)^{-2} = C_0C^2/\lambda^2$.
\end{proof}

The next lemma is another application of Lemma \ref{lem:qcircles}.

\begin{lem}\label{lem:strongalpha2}
Suppose that $\G$ is a $K$-quasiarc, $\d\in (0,1)$ and $\P$ a $\d$-partition of $\G$ with $|\P| \leq N_0/\d$ for some $N_0>1$. There exists $N_1$ depending on $K,N_0$ such that for all $\d' \in [\d,1)$ and all 
$\d'$-partitions $\P'$ of $\G$ we have $|\P'|\leq N_1 /\d'$.
\end{lem}

\begin{proof}
We prove the lemma for $N_1 = 5M_1N_0$ where $M_1$ is as in the third part of Lemma \ref{lem:Mindex}. Contrary to the claim, suppose that there exist a number $\d' \in [\d,1) $ and a $\d'$-partition 
$\P' = \{\G_1,\dots,\G_N\}$ such that $N > N_1/\d'$. For each $i=1,\dots,N$ consider a $\frac{\d\diam{\G}}{\diam{\G_i}}$-partition $\P_i$ of $\G_i$; for those $i$ that $\diam{\G_i} \leq \d\diam{\G}$ set 
$\P_{i} = \{\G_i'\}$. Note that $|\P_i| \geq \frac{\diam{\G_i}}{\d\diam{\G}}$ and that $\P^* = \bigcup_{i=1}^N \P_i$ is a $\d$-partition of $\G$. By Lemma \ref{lem:Mindex},
\[ M(\G,\P^*) \geq (|\P_1|+\dots+|\P_N|)\frac{\d}{2} > M_1N_0.\]
But then $M(\G,\P^*)/M(\G,\P) > M_1$ which is a contradiction.
\end{proof}

The following lemma is used in the proof of Theorem \ref{thm:assouad}.

\begin{lem}\label{lem:tree}
Let $\G$ be a Jordan arc and $0<\d <\d' <1$. Suppose that each $\d$-partition of $\G$ has at least $N$ elements for some $N>1$. Then, there exists 
a subarc $\G' \subset \G$ and a $\d'$-partition $\P'$ of $\G'$ with 
$|\P'| \geq N^{\frac{\log{\d'}+1}{\log{(\d\d'/2)}}}$.
\end{lem}

\begin{proof}
Contrary to the lemma, suppose that for each subarc $\sigma\subset\G$, every $\d'$-partition $\P$ of $\sigma$ has $|\P| \leq  N^{\frac{\log{\d'}+1}{\log{(\d\d'/2)}}}$. We construct a partition of $\G$ 
as follows.

Let $k_0 \in \N$ be the smallest integer that is greater than $\frac{\log{(\d/2)}}{\log{\d'}}$. Let $\P_0 = \{\G_1,\dots,\G_{N_0}\}$ be a $\d'$-partition of $\G$. For each $i=1,\dots,N_0$ let 
$\P_i = \{\G_{i1},\dots,\G_{iN_{i}}\}$ be a $\d'$-partition of $\G_i$. Inductively, suppose that $\G_w$ has been defined where $w=i_1 i_2\dots i_k$, $i_j\in\N$ and $k<k_0$. Then let 
$\P_w = \{\G_{w1},\dots,\G_{wN_{w}}\}$ be a $\d'$-partition of $\G_w$. 

Define $\P^* = \bigcup_{w}\P_w$ where the union is taken over all words $w = i_1\dots i_{k_0}$. Note that for each $\G_w \in \P^*$,
\[ \left ( \frac{\d'}{2} \right )^{k_0} \leq \frac{\diam{\G_w}}{\diam{\G}} \leq (\d')^{k_0} < \frac{\d}{2}.\]
Moreover, by our assumptions, for all partitions $\P_w$ defined,
\[ |\P_w| \leq  N^{\frac{\log{\d'}+1}{\log{(\d\d'/2)}}}. \]
Consequently, since $k_0 \leq  \frac{\log{(\d/2)}}{\log{\d'}} +1$,
\[ |\P^*| \leq \left ( N^{\frac{\log{\d'}+1}{\log{(\d\d'/2)}}}\right )^{k_0} < \left ( N^{\frac{\log{\d'}}{\log{(\d\d'/2)}}}\right )^{k_0} \leq N.\]

If $\P$ is a $\d$-partition of $\G$ then, for each $\G' \in \P$ and each $\G_w \in\P^*$, $\diam{\G_w} < \diam{\G'}$. Thus, $|\P| \leq |\P^*| < N$ which is a contradiction.
\end{proof}

\subsection{A weak chord-arc property}\label{sec:alpha-prop}

Note that, if a curve $\G$ is a chord-arc curve then there exists $M_0>1$ such that for all subarcs $\G'\subset\G$, all $\d>0$ and all $\d$-partitions $\P$ of $\G'$, we have $M(\G',\P)<M_0$.

\emph{A curve or arc $\G$ is said to have the weak chord-arc property if there exists $M_0>0$ such that for all subarcs $\G' \subset \G$ with $\diam\G'<1$, there exists a 
$\diam{\G'}$-partition $\P$ of $\G'$ satisfying $M(\G',\P) \leq M_0$.}

In other words, a curve $\G$ is a weak chord-arc curve if there exists $M_0\geq 1$ such that any subarc $\G'$ of $\G$ can be partitioned to at most $2M_0/\diam{\G'}$ subarcs $\G_1,\dots,\G_N$ of 
diameters comparable to $(\diam{\G'})^2$. It is clear that a chord-arc curve is a weak chord-arc curve but the converse fails; see Section \ref{sec:logn}. The third claim of Lemma \ref{lem:Mindex} implies that, for 
quasicircles, the weak chord-arc property is equivalent to a stronger condition.

\begin{lem}\label{lem:strongalpha}
Suppose that $\G$ is a $K$-quasicircle that satisfies the weak chord-arc property with constant $M_0>1$. For each $\a \in (0,1)$ there exists $M_{\a}>1$ depending only on $K,M_0,\a$ such that for all subarcs 
$\G' \subset \G$ of diameter less than $1$ and all $(\diam{\G'})^{\frac{1}{\a} -1}$-partitions $\P$ we have $M(\G',\P) \leq M_{\a}$.
\end{lem}

\begin{proof}
If $\a \in [\frac{1}{2},1)$ then the claim follows immediately from Lemma \ref{lem:strongalpha2} with $M_{\a} = 10M_0M_1$ where $M_1$ is as in Lemma \ref{lem:Mindex}. 

Suppose now that $\a\in (0,\frac{1}{2})$ and let $k_0$ be the smallest integer with $2^{-k_0}\leq \a$. 
Fix a subarc $\G_0 \subset \G$ with $\diam{\G_0} < 1$ and a $\diam{\G_0}$-partition $\P_0 = \{\G_1,\dots,\G_{N_0}\}$ of $\G_0$ with $|\P_0| \leq 2M_0$. For each $i=1,\dots,N_0$ let $\P_i$ be a 
$\frac{(\diam{\G_0})^4}{\diam{\G_i}}$-partition of $\G_i$. The weak chord-arc property of $\G_i$ and Lemma \ref{lem:strongalpha2} imply $|\P_i| \leq 10M_0M_1(\diam{\G_0})^{-2}$. 
Inductively, suppose that $\G_w$ has been defined where $w=i_1\cdots i_k$, $i_j \in\N$ and $k<k_0$. 
Then, let $\P_w = \{\G_{w1},\dots,\G_{wN_w}\}$ be a $\frac{(\diam{\G_0})^{2^{k+1}}}{\diam{\G_w}}$-partition of $\G_w$. Again, note that $|\P_w| \leq 10M_0M_1(\diam{\G_0})^{-2^k}$.

Define $\P^* = \bigcup_w\P_w$ where the union is taken over all words $w=i_1\cdots i_{k_0}$ constructed as above. Note that $\P^*$ is a $(\diam{\G_0})^{2^{k_0}-1}$-partition of $\G_0$ with 
\[ M(\G_0,\P^*) \leq  \frac{(\diam{\G_0})^{2^{k_0}}}{\diam{\G_0}} \sum_{w=i_1\cdots i_{k_0}}|\P_w| \leq (10M_0M_1)^{k_0}. \]
Since $(\diam{\G_0})^{1/\a} \geq (\diam{\G_0})^{2^{k_0}} \geq \diam{\G_w}$ for all $\G_w \in \P^*$, it follows from Lemma \ref{lem:strongalpha2} that every $(\diam{\G'})^{\frac{1}{\a} -1}$-partition $\P$ of $\G_0$ 
satisfies $M(\G_0,\P) \leq M_{\a}$ with $M_{\a}= (10M_0M_1)^{k_0+1}$.
\end{proof}

By interchanging the roles of $\frac{1}{2}$ and $\a$, and applying the arguments in the proof of Lemma \ref{lem:strongalpha}, the following converse can be obtained.

\begin{rem}\label{rem:strongalpha}
Suppose that $\G$ is a $K$-quasicircle and that there exist $\a\in(0,1)$ and $M_{\a}>1$ such that every subarc $\G' \subset \G$ with $\diam{\G'} < 1$ has a $(\diam{\G'})^{\frac{1}{\a} -1}$-partition $\P$ satisfying 
$M(\G',\P) \leq M_{\a}$. Then $\G$ has the weak chord-arc property.
\end{rem}

\section{Proofs of the main results}\label{sec:mainresults}

The proof of Theorem \ref{thm:main} is given in Section \ref{sec:alpha2reg} and Section \ref{sec:LQC&Vais}. By the theorem of Bonk and Kleiner, it is clear that \emph{(2)} implies \emph{(1)}. 

To show that \emph{(3)} implies \emph{(2)} we first show in Lemma \ref{lem:qcirclenec} that if $\Omega$ satisfies the LQC property then $\S(\Omega,t^{\a})$ is LLC for all $\a\in(0,1)$. Then, in Proposition \ref{prop:2reg} 
we prove that if $\Omega$ has the LQC property and $\partial\Omega$ is a weak chord-arc curve then $\S(\Omega,t^{\a})$ is $2$-regular. 

To prove that \emph{(1)} implies \emph{(3)}, we show in Lemma \ref{lem:qcirclenec} that if $\S(\Omega,t^{\a})$ is LLC for some $\a\in(0,1)$ then $\Omega$ satisfies the LQC property. Then, in Proposition \ref{prop:N(w,h)} we 
show that if $\Omega$ has the LQC property and $\S(\Omega,t^{\a})$ is quasisymmetric to $\mathbb{S}^2$ then $\partial\Omega$ is a weak chord-arc curve.

The proof of Theorem \ref{thm:assouad} is given in Section \ref{sec:assouad}.

\subsection{Ahlfors $2$-regularity}\label{sec:alpha2reg}

The following proposition connects the weak chord-arc property of $\partial\Omega$ with the $2$-regularity of $\S(\Omega,t^{\a})$.

\begin{prop}\label{prop:2reg}
Suppose that $\Omega$ has the LQC property, $\partial\Omega$ has the weak chord-arc property and $\a\in(0,1)$. Then, $\S(\Omega,t^{\a})$ is $2$-regular.
\end{prop}

For the rest of Section \ref{sec:alpha2reg} we assume that there exist $C_0>1$ and $\e_0\in(0,\frac{1}{8C_0})$ such that for any $\e \in [0,\e_0]$, the level set $\g_{\e}$ is a quasicircle and satisfies (\ref{eq:3pts}) for 
some $C_0>1$. To show (\ref{eq:2regdefn}) we first apply some reductions on $a\in\S(\Omega,t^{\a})$ and $r>0$.

\medskip

\emph{Reduction 1.} Let $\S(\Omega,t^{\a})^{+} = \S(\Omega,t^{\a})\cap\R^3_+$. The symmetry of $\S(\Omega, t^{\a})$ with respect to $\R^2\times \{0\}$ implies that it is enough to verify
\begin{equation}\label{eq:Ahlforsreg}
C^{-1} r^2  \le \mathcal{H}^2(B^3(a,r)\cap\S(\Omega,t^{\a})^+) \leq C r^2,
\end{equation}
for some $C>1$ and for all $a\in \S(\Omega,t^{\a})^+$ and $r>0$ sufficiently small.

\medskip

\emph{Reduction 2.} We claim that it is enough to verify (\ref{eq:Ahlforsreg}) only for those points $a\in \S(\Omega,t^{\a})^+$ whose projection satisfies $\dist(\pi(a),\partial\Omega)\leq\e_0$ and for $r<\e_0/3$. 
Indeed, following the notation in \cite{VW2}, let $\D_{\e_0/3}$ be the set of all points in $\Omega$ whose distance from $\partial \Omega$ is greater than  $\e_0/3$ and  $\D_{\e_0/3}^+$ be the subset of 
$\S(\Omega,t^{\a})^+$ whose projection on $\R^2\times\{0\}$ is $\D_{\e_0/3}$. Since $\partial\D_{\e_0/3} = \g_{\e_0/3}$ is a quasicircle, the domain $\D_{\e_0/3}$ is $2$-regular. Therefore, the surface $\D_{\e_0/3}^+$, which is 
the graph of a Lipschitz function on  $\D_{\e_0/3}$, is $2$-regular as well. Thus, if $a\in \S(\Omega,t^{\a})^+$ satisfies $\dist(\pi(a),\partial\Omega)>\e_0$ and $r\in(0,\e_0/3)$ then 
$B^3(a,r)\cap\S(\Omega,t^{\a})^+ \subset \D_{\e_0/3}^+$ and $\mathcal{H}^2(B^3(a,r)\cap\S(\Omega,t^{\a})^+)\simeq r^2$. For the rest we require that $0< r \leq \min\{\e_0/3 ,\frac{1}{180 C_0^2}\diam{\Omega}\}$.

\medskip

\emph{Reduction 3.} It is enough to check the Hausdorff $2$-measure of certain surface pieces on $\S(\Omega,t^{\a})$ defined as follows. Suppose that $x_1,y_1,x_2,y_2$ are points in $\S(\Omega,t^{\a})^+$ satisfying the 
following properties:
\begin{enumerate}
\item[(i)] $\pi(x_1), \pi(y_1)\in \g_{t_1} \text{ and } \pi(x_2), \pi(y_2) \in \g_{t_2} \text{ for some }0\leq t_2<t_1\leq \e_0,$
\item[(ii)] $|\pi(x_1) - \pi(x_2)| = \dist(\pi(x_1),\g_{t_2})=  |\pi(y_1)-\pi(y_2)| =\dist(\pi(y_1),\g_{t_2})=  t_1-t_2,$
\item[(iii)] $\frac{1}{20 C_0}|x_1-y_1| \leq t_1-t_2+ t_1^{\a}- t_2^{\a} \leq \frac{1}{3}|x_1-y_1| \leq \frac{1}{10C_0}\diam{\Omega}$,
\end{enumerate}
Property (i) implies that for each $i=1,2$, $x_i$ is on the same horizontal plane as $y_i$. Property (ii) implies that $\pi(x_2),\pi(y_2)$ are the points of $\g_{t_2}$ which are closest to 
$\pi(x_1),\pi(y_1)$ respectively. Property (iii) implies that $|x_i-y_j|$, $|x_1-x_2|$ and $|y_1-y_2|$ are all comparable and sufficiently small. Denote with $D = D(x_1,y_1,x_2,y_2)$ the piece on $\S(\Omega,t^{\a})^+$ whose 
projection on $\R^2$ is the quadrilateral bounded by $[\pi(x_1),\pi(x_2)]$, $\g_{t_1}(\pi(x_1),\pi(y_1))$, $[\pi(y_1),\pi(y_2)]$, $\g_{t_2}(\pi(x_2),\pi(y_2))$. We call $D$ a \emph{square piece} on $\S(\Omega,t^{\a})^+$.

The following lemma is a corollary of \cite[(6.4)]{VW2} and \cite[Remark 6.1]{VW2}.

\begin{lem}\label{lem:square}
There exist $C_1,C_2 >1$ depending on $C_0,\e_0,\diam{\Omega}$ such that
\begin{enumerate}
\item $\diam{D} \leq C_1|x_1-y_1|$ for every square piece $D = D(x_1,y_1,x_2,y_2)$,
\item for all $a \in \S(\Omega,t^{\a})^+$ and $r>0$ as above, there exist square pieces $D_1$ and $D_2$ such that $1 < \diam{D_2}/\diam{D_1} \leq C_2$ and
\[D_1 \subset B^3(a,r)\cap \S(\Omega,t^{\a})^+ \subset D_2 .\]
\end{enumerate}
\end{lem}

Thus, by (iii), choosing $\e_0$ small enough, we may assume from now on that $\diam{D} \leq (4C_0)^{-1}$ for all square pieces $D$. By Lemma \ref{lem:square} and the discussion above, Proposition \ref{prop:2reg} is now equivalent to the 
following lemma. 

\begin{lem}\label{lem:2regsquare}
There exists $C>1$ depending on $C_0$ and the weak chord-arc constant $M_0$ such that
\begin{equation}\label{eq:2regsquare}
C^{-1}(\diam{D})^2 \leq \mathcal{H}^2(D) \leq C(\diam{D})^2
\end{equation}
for all square pieces $D$ on $\S(\Omega,t^{\a})^+$ defined as above.
\end{lem}

The lower bound in (\ref{eq:2regsquare}) was shown in \cite[Section 6.2.1]{VW2}. For the upper bound, the following lemma is used; we only give a sketch of its proof since it is similar to the 
discussion in \cite[Section 6.2.1]{VW2}.

\begin{lem}\label{lem:estimation}
Let $S$ be a closed subset of $\S(\Omega,t^{\a})^+$ and $0\leq t_2\leq t_1$ be such that $\pi(S)$ intersects with $\g_t$ if and only if $t\in[t_2,t_1]$. Suppose that, for all $t,t'\in[t_2,t_1]$, the Hausdorff distance 
between $\pi(S)\cap\g_t$ and $\pi(S)\cap\g_{t'}$ is less than $c_1|t-t'|$ for some $c_1>1$, and $\pi(S)\cap\g_t$ is a $c_2$-chord-arc curve for some $c_2>1$. Then $\mathcal{H}^2(S) \leq C (\diam{S})^2$ 
for some $C$ depending on $c_1,c_2$.
\end{lem}

\begin{proof}
Fix $\e>0$. Let $t_2 = \tau_1 < \dots < \tau_N = t_1$ be such that the sets $S_i =S \cap (\g_{\tau_i}\times \{\tau_i^{\a}\})$ satisfy $\e/4\leq \dist(S_{i},S_{i+1}) \leq \e/2$. By the first assumption 
of the lemma, it is straightforward to check that $N \leq N_1\diam{S}/\e$ for some $N_1$ depending on $c_1$. The second assumption implies that each $S_i$ contains points $x_{i,1},\dots,x_{i,n_i}$ 
satisfying $|x_{i,j}-x_{i,j+1}| \leq \e/2$ and $n_i \leq N_2 \diam{S_i}/\e \leq N_2\diam{S}/\e$ for some $N_2$ depending on $c_2$. Thus, $S$ can be covered by at most $N_1N_2(\diam{S})^2/\e^2$ balls of 
radius $\e$ and the lemma follows.
\end{proof}

For the upper bound of (\ref{eq:2regsquare}) we consider two cases. The first case is an application of Lemma \ref{lem:estimation} while in the second case we use the weak chord-arc condition to 
subdivide $D$ into smaller pieces on which the first case applies. 

\begin{proof}[{Proof of Lemma \ref{lem:2regsquare}}]
As mentioned already, the lower bound of (\ref{eq:2regsquare}) follows from the discussion in \cite[Section 6.2.1]{VW2} even without the weak chord-arc assumption of $\partial\Omega$. It remains to show the upper 
bound. By Lemma \ref{lem:strongalpha2}, there exists $M_{\a}>1$ such that for any subarc $\G \subset \partial\Omega$ with $\diam{\G} < 1$ and for any $(\diam{\G})^{\frac{1}{\a}-1}$-partition $\P$ of $\G$ we have 
$M(\G,\P) \leq M_{\a}$.

Fix a square piece $D = D(x_1,y_1,x_2,y_2)$ where $x_1,y_1,x_2,y_2$ satisfy equations (i)-(iii). Let $x_0,y_0 \in \partial\Omega$ be such that
\[ |x_0-\pi(x_1)| = |y_0-\pi(y_1)| = t_1 \text{ and } |x_0-\pi(x_2)| = |y_0-\pi(y_2)| = t_2\]
and set $\G_0$ to be the subarc of $\partial\Omega$, of smaller diameter, with endpoints $x_0,y_0$. The choice of $\e_0$ implies $\diam{\G_0} \leq C_0|x_0-y_0| \leq C_0(\diam{D} + 2t_1)\leq C_0(\diam{D} + 2\e_0) < 1/2$.

\medskip

\emph{Case 1.} Suppose that $(\diam{\G_0})^{\frac{1}{\a}} \leq t_2/10$. We claim that there is $C>1$ depending on $C_0,M_{\a}$ such that for each $\e \in [t_2,t_1]$, the arc $\g_{\e}\cap \pi(D)$ is a $C$-chord-arc curve. 
Assuming the claim, the upper bound of (\ref{eq:2regsquare}) follows from Lemma \ref{lem:estimation}. To prove the claim, it suffices to show that $\g_{t_2}\cap \pi(D)$ is a $C$-chord-arc curve. 
Then, using the fact that for each $\e>t_2$, the set $\g_{\e}$ is the $(\e-t_2)$-level set of $\g_{t_2}$, the claim follows from Lemma \ref{lem:levelcathm}. Let $\s$ be a subarc of $\g_{t_2}\cap D$ and $x,y$ 
be the endpoints of $\s$. Since $\g_{t_2}$ is a quasicircle, it suffices to show that there exists $C>1$ such that $\ell(\s) \leq C\diam{\s}$. Let $x',y' \in \G_0$ be such that $|x-x'|=|y-y'| = t_2$ and 
set $\s' = \G_0(x',y')$. 

If $\diam{\s'} \leq t_2/10$ then, by Lemma \ref{lem:levelcalemma}, $\s$ is a $c_0$-chord-arc curve for some universal $c_0>1$. 

If $\diam{\s'} \in [t_2/10,10t_2]$ then apply Lemma \ref{lem:qcircles} to get a $10^{-2}$-partition $\P = \s'_1,\dots,\s'_N$ of $\s'$ with $N$ bounded above by a positive constant depending on $C_0$. 
For each $n=1,\dots,N$ let $\s_n$ be the set of all points $z$ in $\s$ that satisfy $\dist(z,\s_n') = t_2$. By Lemma \ref{lem:levelsubarc}, $\{\s_n\}$ are subarcs of $\s$, perhaps not mutually disjoint. Thus, $\s$ is a 
union of $N$ $c_0$-chord-arc curves, hence $\ell(\sigma) \leq c_1\diam{\s}$ for some $c_1>1$ depending on $C_0$.

Finally, if $\diam{\s'} > 10t_2$ note that $\diam{\s}\simeq \diam{\s'}$. Since $(\diam{\s'})^{\frac{1}{\a}} \leq (\diam{\G_0})^{\frac{1}{\a}} \leq t_2/10$, Lemma \ref{lem:strongalpha2} and the weak chord-arc property of $\s'$ 
imply that there exists a $(\frac{t_2}{10\diam{\s'}})$-partition $\P = \{\s'_1,\dots,\s'_N\}$ of $\s'$ with $N \leq N_0\diam{\s'}/t_2$ and $N_0$ depending on $C_0,M_{\a}$. Define $\s_1,\dots,\s_N$ as above and note that each 
$\s_n$ is a $c_0$-chord-arc curve and satisfies $\diam{\s_n} \leq \diam{\s_n'} + t_2 \lesssim t_2$. Hence, 
\[ \ell(\s) \leq \sum_{n=1}^{N} \ell(\s_n) \leq \sum_{n=1}^{N} c_1 \diam{\s_n} \lesssim N t_2 \lesssim \diam{\s'} \simeq \diam{\s}.\]

\emph{Case 2.} Suppose that $(\diam{\G_0})^{\frac{1}{\a}} > t_2/10$. Let $H = 10^{2/\a}$. For a subarc $\G' \subset \G_0$ define $ D(\G')$ to be the set of all points $(z,(\dist(z,\partial\Omega))^{\a})$ such that 
$z\in\overline{\Omega}$, $\dist(z,\partial\Omega) = \dist(z,\G')$ and 
\[0 \leq \dist(z,\partial\Omega) \leq H(\diam{\G'})^{1/\a}. \]
Note that $\diam{D(\G')} \leq 2(H^{\a}\diam{\G'} + H(\diam{\G'})^{1/\a} + \diam{\g_{t_1}(x_1,y_1)}) \lesssim \diam{\G'}$. From the middle inequality of (iii), it is easy to see that $t_1  < H(\diam{\G_0})^{1/\a}$ and thus 
$D \subset D(\G_0)$. Let 
\[E_0 = \{x\in D(\G_0) \colon \dist(\pi(x),\partial\Omega) \geq 10(\diam{\G_0})^{2/\a}\}\]
and note that $\diam{E_0}\leq \diam{D(\G_0)}\lesssim \diam{\G_0}$. 

We claim that $\mathcal{H}^2(E_0) \leq C (\diam{\G_0})^2$ for some $C>1$ depending on $C_0,M_{\a}$. As in \emph{Case 1}, it is enough to show that $\g = \g_t \cap \pi(E_0)$ is a chord-arc curve when 
$t = 10(\diam{\G_0})^{\frac{2}{\a}}$. The claim then follows from Lemma \ref{lem:levelcathm} and  Lemma \ref{lem:estimation}. Let $\sigma \subset \g$ and define $\sigma' \subset \G_0$ as in \emph{Case 1}. 
If $\diam{\sigma'} \leq 10t$ then the claim follows from Lemma \ref{lem:levelcalemma} and Lemma \ref{lem:strongalpha2} as in  \emph{Case 1}. If $\diam{\sigma'} \geq 10t$ then let $\{\sigma'_1,\dots,\sigma'_m\}$ be a 
$(\diam{\sigma'})^{\frac{1}{\a}-1}$-partition of 
$\sigma'$ with $m\leq 2M_{\a}$. For each $i=1,\dots,m$ let $\{\sigma'_{i1},\dots,\sigma'_{iN_i}\}$ be a $(t(\diam{\sigma'})^{-\frac{1}{\a}})$-partition of $\sigma_i'$. Define subarcs $\{\sigma_{ij}\}$ of $\sigma$ as in 
\emph{Case 1}. Since $\diam{\sigma'_{ij}}\leq t < 10t$, each $\sigma_{ij}$ is a $c_1$-chord-arc curve for some $c_1>1$ depending on $C_0$. Moreover, by the weak chord-arc property and Lemma \ref{lem:strongalpha}, 
$N_1+\dots+N_m \leq10M_1M_{\a}^2\diam{\sigma'}/t$ where $M_1$ is as in Lemma \ref{lem:Mindex}. Since $\diam{\sigma'} \simeq \diam{\sigma}$,
\[ \ell(\sigma) \leq \sum_{i,j}\ell(\sigma_{ij}) \leq  c_1 \sum_{i,j} \diam{\sigma_{ij}} \lesssim t\sum_{i=1}^m N_i \lesssim \diam{\sigma'} \lesssim \diam{\sigma}.\] 
and the claim follows.

Let $\P_1 = \{\G_1,\dots,\G_N\}$ be a $(\diam{\G_0})^{\frac{1}{\a}-1}$-partition of $\G_0$ with $N \leq 2M_{\a}(\diam{\G_0})^{1-\frac{1}{\a}}$. Define $D(\G_i)$, $E_i$ as above and note that the choice of $H$ yields 
$D(\G_0) \subset E_0 \cup \bigcup_{i=1}^N D(\G_i)$. Similarly as above, $\mathcal{H}^2(E_i) \leq C (\diam{\G_0})^{\frac{2}{\a}}$.

For each $i=1,\dots,N$ let $\{\G_{i1},\dots,\G_{iN_i}\}$ be a $\frac{(\diam{\G_0})^{1/\a^2}}{\diam{\G_i}}$-partition of $\G_i$ and set $\P_2 = \{\G_{ij}\}$. The weak chord-arc 
condition and Lemma \ref{lem:qcircles} imply that $N_i \leq 10M_{\a}M_1 (\diam{\G_i})^{1-\frac{1}{\a}} \leq 10M_{\a}M_1 2^{\frac{1}{\a}-1}(\diam{\G_0})^{\frac{1}{\a}-\frac{1}{\a^2}}$ where $M_1$ is as in Lemma \ref{lem:Mindex}. 
Thus, $|\P_2| \leq (m_0 2^{\frac{1}{\a}-1})^2 (\diam{\G_0})^{1-\frac{1}{\a^2}}$ with $m_0 =10M_{\a}M_1$. Again, the choice of $H$ yields
\[ D(\G_0) \subset E_0 \cup \bigcup_{i=1}^N E_i \cup \bigcup_{i=1}^N\bigcup_{j=1}^{N_i} D(\G_{ij}). \]
Inductively, we obtain partitions $\P_k = \{\G_{i_1\cdots i_k}\}$ with 
\[ |\P_k| \leq (m_0 2^{\frac{1}{\a}-1})^{k}(\diam{\G_0})^{1-\frac{1}{\a^{k}}}  \] 
and square-like pieces $E_{i_1\cdots i_k}$ with
\[ D(\G_0) \subset E_0 \cup \bigcup_i E_i \cup \bigcup_{i_1,i_2}E_{i_1,i_2} \cup \dots \cup \bigcup_{\G_{i_1\cdots i_k}\in \P_k}D(\G_{i_1\cdots i_k})\]
and $\mathcal{H}^2(E_{i_1\cdots i_k}) \leq C(\diam{\G_0})^{\frac{2}{\a^k}}$. Therefore, since $D \subset D(\G_0)$,
\begin{align*}
\mathcal{H}^2(D) &\leq \mathcal{H}^2(E_0) + \sum_{i}\mathcal{H}^2(E_i) + \sum_{i_1,i_2} \mathcal{H}^2 (E_{i_1,i_2}) + \dots \\
                 &\lesssim (\diam{\G_0})^2 \sum_{k=0}^{\infty} (m_0 2^{\frac{1}{\a}-1})^k (\diam{\G_0})^{\frac{1}{\a^k} - 1}\\
                 &\leq (\diam{\G_0})^2\sum_{k=0}^{\infty} \frac{(m_0 2^{\frac{1}{\a}-1})^k}{2^{\frac{1}{\a^k} - 1}}.
\end{align*}  
It is easy to see that the latter series converges. Since $\diam{\G_0} \leq \diam{D}$, we conclude that $\mathcal{H}^2(D) \lesssim (\diam{D})^2$ and the proof is complete.
\end{proof}

\subsection{The LQC property and V\"ais\"al\"a's method}\label{sec:LQC&Vais}

The connection between the LQC property of $\Omega$ and the LLC property of $\S(\Omega,t^{\a})$ is established in the following proposition from \cite{VW2}.

\begin{prop}[{\cite[Proposition 5.1, Lemma 5.6]{VW2}}]\label{prop:LLCLQC}
Suppose that $\a\in(0,1)$ and $\Omega$ is a Jordan domain whose boundary $\partial\Omega$ is a quasicircle. Then $\S(\Omega,t^{\a})$ is LLC if and only if $\Omega$ has the LQC property.
\end{prop}

It turns out that the quasicircle assumption of $\partial\Omega$ can be dropped in Proposition \ref{prop:LLCLQC}.

\begin{lem}\label{lem:qcirclenec}
Suppose that $\Omega$ is a Jordan domain and $\a\in(0,1)$. Then, $\S(\Omega,t^{\a})$ is LLC if and only if $\Omega$ has the LQC property.
\end{lem}

\begin{proof}
Suppose that $\S = \S(\Omega,t^{\a})$ is $\lambda$-LLC for some $\lambda>1$. In view of Proposition \ref{prop:LLCLQC} it suffices to show that $\partial\Omega$ is a quasicircle. In particular, we show
 that $\partial\Omega$ satisfies the $2$-point condition (\ref{eq:3pts}) with $C=4\lambda^2$. 

Let $x,y\in\partial\Omega$ and $\g,\g'$ be the two components of $\partial\Omega\setminus \{x,y\}$. By the $\lambda-\text{LLC}_1$ property, the points $x,y$ are contained in a continuum $E$ in 
$B^3(x,2\lambda|x-y|)\cap\S$. Then the projection $\pi(E)$ is a continuum in $\overline{\Omega}$ containing $x,y$. Suppose that there exist points $z,z'$ in $\g,\g'$ respectively, which lie outside of 
$B^3(x,2\lambda^2|x-y|)$. The $\lambda-\text{LLC}_2$ property implies that there exists a continuum $E' \subset \S\setminus B^3(x,2\lambda|x-y|)$ that contains $z,z'$. But then, the projection $\pi(E')$
 is a continuum in $\overline{\Omega}$ containing $z,z'$. It follows that $\pi(E)\cap\pi(E') \neq \emptyset$ and, since $\S$ is symmetric with respect to $\R^2\times\{0\}$, $E'$ intersects 
$B^3(x,2\lambda|x-y|)$ which is a contradiction. Therefore, at least one of $\g,\g'$ lies in $\overline{B}^3(x,2\lambda^2|x-y|)$ and $\min\{\diam{\g},\diam{\g'}\} \leq 4\lambda^2|x-y|$.
\end{proof}

We now show that the weak chord-arc condition of $\partial\Omega$ is necessary for $\S(\Omega,t^{\a})$ to be quasisymmetric to $\mathbb{S}^2$. This concludes the proof of Theorem \ref{thm:main}. The proof follows closely 
that of \cite[Proposition 6.2]{VW2}. The main idea used is due to V\"ais\"al\"a from \cite{Vais3}.

\begin{prop}\label{prop:N(w,h)}
Suppose that $\Omega$ has the LQC property and $\a\in(0,1)$. If $\S(\Omega,t^{\a})$ is quasisymmetric to $\mathbb{S}^2$ then $\partial\Omega$ is a weak chord-arc curve.
\end{prop}

\begin{proof}
By our assumptions, there exist $\e_0>0$ and $C>1$ such that, for all $\e \in [0,\e_0]$, the set $\g_{\e}$ satisfies (\ref{eq:3pts}) with constant $C$. Set $\partial\Omega = \G$.

Suppose that the claim is false. Then, by Remark \ref{rem:strongalpha}, for each $n\in\N$, there exists a subarc $\G_n \subset \G$, of diameter less than $1$, and a $(\diam{\G_n})^{\frac{1}{\a}-1}$-partition 
$\P_n = \{\G_{n,1}, \dots,\G_{n,N_n}\}$ with $M(\G_n,\P_n) > 4Cn$. By Lemma \ref{lem:Mindex}, the latter implies that
\[ N_n > 4Cn (\diam{\G_n})^{1-\frac{1}{\a}}.\]

Let $\{x_{n,0}, x_{n,1}, \dots, x_{n,N_n}\}$ be the endpoints of the arcs $\G_{n,1}, \dots,\G_{n,N_n}$, ordered consecutively according to the orientation in $\G_n$ with $x_{n,0}, x_{n,N_n}$ being the endpoints of $\G_n$.
By adding more points from each subarc $\G_{n,i}$, in this collection, we may further assume that
\[ \frac{1}{4C}(\diam{\G_n})^{\frac{1}{\a}} \leq |x_{n,i} - x_{n,i-1}| \leq \frac{1}{2C}(\diam{\G_n})^{\frac{1}{\a}}.\]
It follows that
\[ \sum_{i=1}^{N_n} |x_{n,i-1}-x_{n,i}| \geq n \diam{\G_n}.\]
Set $d_n = (10C^2)^{-1}\min_{1\leq i \leq N_n}|x_{n,i}-x_{n,i-1}|$ and note that $d_n^{\a} \simeq \diam{\G_n}$. 

The rest of the proof is identical to that of \cite[Proposition 5.1]{VW2} by setting $\varphi(t) = t^{\a}$ therein. We sketch the remaining steps for the sake of completeness.

Fix $n\in\N$ and write $N_n =N$ and $x_{n,i} = x_i$. Assume that there exists an $\eta$-quasisymmetric map $F$ from $\S(\Omega,t^{\a})$ onto $\mathbb{S}^2$. Composing $F$ with a suitable M\"obius map
 we may assume that $F(\S(\Omega,t^{\a})^+)$ is contained in the unit disc $\B^2$. Since $\partial\Omega$ is a quasicircle, we can find points $w_0,\dots,w_N$ on $\G$ and points $w'_0,\dots,w'_N$ on 
$\g_d =\{x\in\Omega\colon \dist(x,\partial\Omega) = d\}$ which follow the orientation of $\{x_{0},\dots,x_{N}\}$ such that $|w_i-x_i| \in [d,3C_0d]$ and $|w_i-w_i'| = d$. Hence, $|w_i-w_{i+1}| \simeq d$. Let $D$ be the 
square-like piece on $\S(\Omega,t^{\a})^+$ whose projection on $\R^2\times\{0\}$ is the Jordan domain bounded by $\G(w_{0},w_{N})$, $\g_d(w_{0}',w_{N}')$, $[w_{0},w_{0}']$, $[w_{N}),w_{N}']$. 

The partition of $\G(w_0,w_N)$ into the subarcs $\G(w_{i-1},w_i)$, $i=1,\dots,N$ induces a partition of $D$ into $N$ tall and narrow strips $D_i$ with height in the magnitude $d^{\a}$ and width in the magnitude d. Each $D_i$ 
is further partitioned by planes parallel to $\R^2\times\{0\}$ into $k$ square-like pieces $D_{ij}$ with $k \simeq d^{\a-1}$. The quasisymmetry of $F$ implies that $(\diam{F(D_{ij})})^2 \leq c_1\mathcal{H}^2(D_{ij})$ with 
$c_1>1$ depending on $\eta,C$. Summing first over $j$ and then over $i$, and applying H\"older's inequality twice, we obtain $(\diam{F(D)})^2\leq Nkc_1\mathcal{H}^2(D) \leq c_2n\mathcal{H}^2(D)$ with $c_2>1$ depending on 
$\eta,C$. On the other hand, the quasisymmetry of $F$ on $D$ implies that $\mathcal{H}^2(F(D)) \leq c_3(\diam{F(D)})^2$ with $c_3>1$ depending on $\eta,C$. Since $\diam{F(D)}\neq 0$, letting $n\to\infty$, we obtain a contradiction.
\end{proof}

\subsection{Assouad dimension}\label{sec:assouad}

The Assouad dimension of a metric space $(X,d)$, introduced in \cite{Assouad}, is the infimum of all $s>0$ that satisfy the following property: there exists $C>1$ such that for any $Y\subset X$ and 
$\d\in(0,1)$, the set $Y$ can be covered by at most $C\d^{-s}$ subsets of diameter at most $\d\diam{Y}$. In a sense, the main difference between Hausdorff and Assouad dimension of a space $X$ is that that the 
former is related to the average small scale structure of $X$, while the latter measures the size of $X$ in all scales. See \cite{Luukkainen} for a detailed survey of the concept.

\begin{rem}
The 2-point condition (\ref{eq:3pts}) implies that if $X$ is a quasicircle, then all sets in the definition of the Assouad dimension can be replaced by subarcs of $X$.
\end{rem}

The claim of the remark becomes evident after noticing that for all subsets $Y$ of a $K$-quasicircle $\G$ there exists a subarc $\G' \subset \G$ containing $Y$, such that $\diam{\G'} \leq C^{-1}\diam{Y}$ for some $C>1$ 
depending on $K$.

\medskip

We now turn to the proof of Theorem \ref{thm:assouad}.

\begin{proof}[{Proof of Theorem \ref{thm:assouad}}] Suppose that $\G$ is a $K$-quasicircle with Assouad dimension greater than $1$; in particular, greater than $1+\e$ for some fixed $\e\in(0,1)$. We claim that $\G$ does not have 
the weak chord-arc property.

Contrary to the claim, assume that $\G$ satisfies the weak chord-arc condition for some $M_0>1$. By Lemma \ref{lem:strongalpha2} there exists $N_0>1$ depending on $M_0,K$ such that for all $\G' \subset\G$ with $\diam{\G'}<1$ 
and all $\diam{\G'}$-partitions $\P$ of $\G'$ we have $|\P| < N_0$. By our assumption on the Assouad dimension of $\G$, there exists a subarc $\G' \subset \G$ and a number $\d\in (0,1)$ such that all $\d$-partitions $\P$ of 
$\G'$ satisfy $|\P| \geq M\d^{-1-\e}$ where $M = 10N_0M_1M_{\a}$, $M_1>1$ is as in the third claim of Lemma \ref{lem:Mindex}, $M_{\a}>1$ is the number in Lemma \ref{lem:strongalpha} associated to 
$\a = \frac{1-\sqrt{\beta}}{1+\sqrt{\beta}}$ and $\beta = \frac{1+\e/2}{1+\e}$. The subarc $\G'$ can be chosen small enough so that $\diam{\G'} < \min\{(2M_0M_1)^{-2/\e},e^{-1/(1-\sqrt{\beta})}\}$.

\emph{Case 1.} Suppose that $\diam{\G'} \leq 2\d^{(1-\sqrt{\beta})/\sqrt{\beta}}$. Assume first that $\d > \diam{\G'}$. Let $\P$ be a $\d$-partition of $\G'$. The weak chord-arc property of $\G'$ and Lemma 
\ref{lem:strongalpha2} yield $|\P| \leq N_0\d^{-1} < M \d^{-1-\e}$ which is false.

Assume now that $\d \leq \diam{\G'}$. Let $\a' \in (0,1)$ be such that $(\diam{\G'})^{\frac{1}{\a'}-1} = \d$. The assumption on $\diam{\G'},\d$ and the fact that $\diam{\G'} < 1/2$ yield 
$\a' \geq \frac{1-\sqrt{\beta}}{1+\sqrt{\beta}} = \a$. By Lemma \ref{lem:strongalpha2} and Lemma \ref{lem:strongalpha}, every $\d$-partition $\P$ of $\G'$ satisfies $|\P| \leq 10M_{\a}M_1 \d^{-1} < M \d^{-1-\e}$ which is also 
false.

\emph{Case 2.} Suppose that $\diam{\G'} > 2\d^{(1-\sqrt{\beta})/\sqrt{\beta}}$. By our assumptions on $\d$ and $\diam{\G'}$ we have
\begin{align*}
\frac{\log{(\diam{\G'})}+1}{\log{(\frac{\d}{2}\diam{\G'})}} &= \frac{\log{(\diam{\G'})}+1}{\log{(\diam{\G'})}}\frac{\log{\d}}{\log{(\frac{\d}{2}\diam{\G'})}}\frac{\log{(\diam{\G'})}}{\log{\d}}\\
&\geq \beta \frac{\log{(\diam{\G'})}}{\log{\d}}.
\end{align*}
Apply Lemma \ref{lem:tree} for $\G'$ with $\d' = \diam{\G}$ and $N = \d^{-1-\e}$. There exists a subarc 
$\G'' \subset \G'$ and a $\diam{\G'}$-partition $\P'$ such that
\[ |\P'| \geq \d^{(-1-\e)\beta \frac{\log{(\diam{\G'})}}{\log{\d}}}  \geq (\diam{\G'})^{-(1+\e/2)}.\]

We create a $\diam{\G''}$-partition of $\G''$ as follows. For each $\sigma\in \P'$ let $\P(\sigma)$ be a $\frac{(\diam{\G''})^2}{\diam{\sigma}}$-partition of $\sigma$; for those $\sigma \in \P'$ that satisfy 
$\diam{\sigma} < (\diam{\G''})^2$ set 
$\P(\sigma) =\{\sigma\}$. Define $\P''$ to be the union of all partitions $\P(\sigma)$. It is easy to see that $\P''$ is a $\diam{\G''}$-partition of $\G''$ and Lemma \ref{lem:Mindex} gives
\[ M(\G'',\P'') \geq M(\G'',\P') \geq \frac{1}{2} (\diam{\G'})^{-\e/2} > M_0M_1. \]
The latter, however, is false by Lemma \ref{lem:Mindex} and the proof is complete.
\end{proof}

\section{Examples from homogeneous snowflakes}\label{sec:snowflakes}

Let $N \geq 4$ be a natural number and $p \in (1/4,1/2)$. A homogeneous $(N,p)$-snowflake is constructed as follows. Let $S_0$ be a regular $N$-gon, of diameter equal to $1/2$. At the $n$-th step, 
the polygon $S_{n+1}$ is constructed by replacing \emph{all} of the $N4^n$ edges of $S_{n}$ by the \emph{same} rescaled and rotated copy of one of the two polygonal arcs of Figure \ref{fig:figure1}, in such a way that the 
polygonal regions are expanding. The curve $\mathcal S$ is obtained by taking the limit of $S_{n}$, just as in the construction of the usual von Koch snowflake. It is easy to verify that every homogeneous snowflake
satisfies (\ref{eq:3pts}) for some $C$ depending on $N,p$, and as a result is a quasicircle.

\begin{figure}[ht]
\includegraphics[scale=1.7]{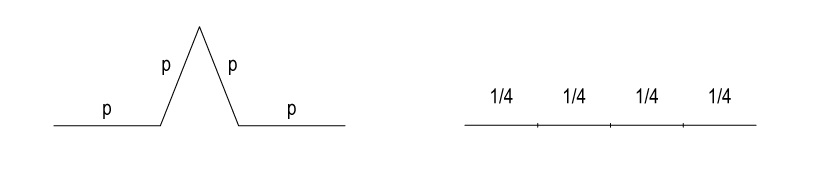}
\caption{}
\label{fig:figure1}
\end{figure}

Let $E$ be an edge of some $S_n$ towards the construction of $\mathcal{S}$. Denote with $\mathcal{S}_E$ the subarc of $\mathcal{S}$, of smaller diameter, having the same endpoints as $E$. The next lemma will ease some 
of the computations in the rest.

\begin{lem}\label{lem:edges}
A homogeneous $(N,p)$-snowflake $\mathcal S$ is a weak chord-arc curve if and only if there exists $M^*>1$ such that every subarc $\mathcal{S}_E$ has a $\diam{\mathcal{S}_E}$-partition $\P$ with 
$M(\mathcal{S}_E,\P) \leq M_0$. 
\end{lem}

\begin{proof}
The necessity is clear. For the sufficiency, fix a subarc $\G \subset \mathcal{S}$ and let $n_0$ be the greatest integer $n$ for which $\G$ is contained in $\mathcal{S}_E$ for some edge $E$ of $S_n$. 
Assume for now that $n_0>0$. Denote by $E_i$, $i=1,\dots,4$, the oriented four segments constructed after $E$ in the $n_0+1$ step, that is $\mathcal{S}_E = \bigcup_{i=1}^4\mathcal{S}_{E_i}$. 
Inductively, if $w=i_1\cdots i_k$ with $i_j \in \{1,\dots,4\}$, let $E_{wi}$, be the oriented four segments constructed after $E_w$ in the $n_0+k+1$ step.

Suppose that $\G$ contains a subarc $\mathcal{S}_{E_i}$, $i=1,\dots,4$. Then $\diam{\G} \geq \frac{1}{4}\diam{\mathcal{S}_E} = \frac{1}{4}\diam{E}$. Let $\P$ be a $(\frac{1}{4}\diam{E})$-partition of $\mathcal{S}_E$ and 
$\P'=\{\G'\cap\G \colon \G'\in\P\}$. The weak chord-arc property of $\mathcal{S}_E$ and Lemma \ref{lem:Mindex} imply that there exists $N_0>0$ such that $|\P'| \leq |\P| \leq N_0(\diam{\G})^{-1}$. Since 
$\diam{\G'} \leq (\diam{\G})^2$ for each $\G'\in\P'$ it follows that $\G$ has the weak chord-arc property for some $M^*$ depending on $N_0$.

Suppose now that $\G$ contains none of the $\mathcal{S}_{E_i}$, $i=1,\dots,4$. Then, there exist $i\in\{1,2,3\}$ and maximal integers $r,q \geq 0$ such that 
$\G \subset \mathcal{S}_{E_{i4^q}}\cup \mathcal{S}_{E_{(i+1)1^r}}$. In this case, apply the arguments above for $\G\cap \mathcal{S}_{E_{i4^q}}$ and $\G\cap \mathcal{S}_{E_{(i+1)1^r}}$.

If $n_0 = 0$ then apply the arguments above for $\G\cap\mathcal{S}_{E_1},\dots,\G\cap\mathcal{S}_{E_N}$ where $E_1,\dots,E_N$ are the edges of $S_0$.
\end{proof}

\subsection{A non-rectifiable Jordan curve that satisfies the weak chord-arc property}\label{sec:logn}

Let $\mathcal S$ be the homogeneous $(N,p)$-snowflake where the first polygonal arc in Figure \ref{fig:figure1} is used only at the $10^n$-th steps, $n\in\N$. We also require that $\diam{\mathcal{S}} < (4p)^{-1}$.
Note that at the $10^k$ step of construction, the length of the polygonal curve $S_{10^k}$ is equal to $(4p)^{k}$. Thus, $\mathcal{S}$ is not rectifiable. We claim that $\mathcal{S}$ has the weak chord-arc property.

By Lemma \ref{lem:edges}, it suffices to check that all subarcs $\mathcal{S}_E$ are weak chord-arc curves. Fix an edge $E$ built at step $n$. Then $\diam{E} \geq 4^{-n}$. Let $k_0$ be the smallest $k\in\N$ such that 
$\diam{E'} \leq (\diam{E})^2$ for all $E'\in S_{n+k}$. We claim that $k_0\leq 9n$. Indeed, the construction of $\mathcal{S}$ implies that at step $10n$, each edge has diameter equal to 
$(4p)4^{-9n}\diam{E} \leq (4p)(\diam{E})^{10}$ since the first polygonal arc in Figure \ref{fig:figure1} has been used only once. The claim follows from our assumption that $\diam{E} \leq \diam{\mathcal{S}} < (4p)^{-1}$.

Let $\P$ be the set of all subarcs $\mathcal{S}_{E'}$ where $E'$ are constructed at step $n+k_0$ and have $E$ as their common parent. Then, $\diam{E'} = 4^{-k_0}C\diam{E}$ with $C=4p$ if the first polygonal arc has been used 
in the $k_0$ steps or $C=1$ otherwise. Since $k_0$ is minimal, 
\[\frac{1}{4}(\diam{\mathcal{S}_E})^2 \leq \diam{\mathcal{S}_{E'}}  \leq (\diam{\mathcal{S}_E})^2.\] 
Therefore, $|\P| = 4^{k_0} \leq 4C(\diam{E})^{-1} = 4C(\diam{\mathcal{S}_E})^{-1}$. By Lemma \ref{lem:qcircles}, there exists a $\diam{\mathcal{S}_E}$-partition of $\mathcal{S}_E$ that has at most $C'(\diam{\mathcal{S}_E})^{-1}$ 
elements, for some $C'>1$ depending on $N,p$. Hence, $\mathcal{S}_E$ has the weak chord-arc property.

\begin{cor}\label{cor:logn}
There exists a Jordan domain $\Omega$ with nonrectifiable boundary such that $\S(\Omega,t^{\a})$ is a quasisymmetric sphere for all $\a\in(0,1]$.
\end{cor}

\begin{proof}
It follows from the discussion in Section 7 of \cite{VW} that there exists $p_0\in(\frac{1}{4},\frac{1}{2})$ and an integer $N_0 \geq 4$ such that every homogeneous $(N,p)$-snowflake with $p\leq p_0$, $N\geq N_0$ 
bounds a domain that satisfies the LQC property. Let $\Omega$ be the domain bounded by the snowflake constructed above with $p\leq p_0$, $N\geq N_0$. Since $\mathcal{S}$ is a quasicircle, $\S(\Omega,t)$ is a
quasisymmetric sphere. Moreover, the weak chord-arc property of $\mathcal{S}$ and Theorem \ref{thm:main} imply that $\S(\Omega,t^{\a})$ is a quasisymmetric sphere for all $\a\in(0,1)$.
\end{proof}

\subsection{A quasicircle of Assouad dimension $1$ that does not satisfy the weak chord-arc property}\label{sec:sqrtn}

Let $\mathcal S$ be the homogeneous $(N,p)$-snowflake where the first polygonal arc in Figure \ref{fig:figure1} is used only at the $n^2$-th steps for $n\in\N$. For convenience we also assume that each edge of $S_0$ has 
length equal to $1$. Then, if $E$ is an edge of $S_n$, $\diam{E} = 4^{-n}(4p)^{\lfloor \sqrt{n}\rfloor}$ where, $\lfloor x\rfloor$ denotes the greatest integer which is smaller than $x$.

We show that $\mathcal{S}$ has Assouad dimension equal to $1$ but does not have the weak chord-arc property.

Fix $\epsilon >0$; we claim that $\mathcal{S}$ has Assouad dimension less than $1+\epsilon$. Similarly to Lemma \ref{lem:edges}, it is easy to show that it is enough to verify the Assouad condition 
only for the subarcs $\mathcal{S}_E$. Take $\delta\in(0,1)$ and an edge $E$ of the $n$-th step polygon $S_n$, for some $n \in \N$. Let $m$ be the largest integer such that $\diam{E'}\geq \d\diam{E}$ for all edges $E'$ of $S_m$
and $\P$ be the set of all subarcs $\mathcal{S}_{E'}\subset \mathcal{S}_{E}$ where $E'$ is an edge of $S_m$. Then, $4^{n-m}(4p)^{\lfloor \sqrt{m}\rfloor-\lfloor \sqrt{n}\rfloor} \geq \delta$ and
\begin{equation}\label{eq:homexample1}
(m-n)\log{4} - (\sqrt{m} - \sqrt{n})\log{4p} \leq -\log{\d} + \log{4p}
\end{equation}
By elementary calculus, there exists $M>0$ depending on $\e,p$ such that 
$\sqrt{x} - \sqrt{y} \leq \frac{\e}{1+\e}\frac{\log{4}}{\log{4p}}(x-y)$ for all $x > M$ and $0<y<x$. 
If $m<M$ then clearly $|\P| = 4^{m-n} \leq 4^M < 4^{M}\d^{-1-\epsilon}$. If $m \geq M$ then $(m-n)\log{4} - (\sqrt{m} - \sqrt{n})\log{4p} \geq \frac{\log{4}}{1+\e}(m-n)$ 
and by (\ref{eq:homexample1})
\[ (m-n)\log{4} \leq \log\d^{-1-\epsilon} + 2\log{4p}.\]
Therefore, $|\P| = 4^{m-n} \leq 4^{M}(4p)^2\d^{-1-\epsilon}$
and the claim follows. Since $\e$ was chosen arbitrarily, $\mathcal{S}$ has Assouad dimension equal to $1$.

We show now that $\mathcal{S}$ does not have the weak chord-arc property. Let $n\in\N$, $E$ be an edge of the $n$-th step polygon $S_n$ and $m\geq n$ be the greatest integer such that 
$\diam{E'}\geq (\diam{E})^2$ for each edge $E'$ of $S_m$. Let $\P$ be the set of all subarcs $\mathcal{S}_{E'}\subset\mathcal{S}_E$ where $E'$ are edges of $S_m$. Then,
\[\left(4^{-n}(4p)^{\sqrt{n}-1}\right)^2 \leq (\diam{E})^2 \leq \diam{E'} \leq 4^{-m}(4p)^{\sqrt{m}} \]
which yields
\[  \frac12 - \frac{n}{m} \leq \frac{\log(4p)}{\sqrt{m}\log{4}}\left( \frac12 - \sqrt{\frac{n}{m}} \right) + \frac{\log{4p}}{m\log{4}} .\]
Note that as $n$ goes to infinity, $m$ goes to infinity and $n/m$ goes arbitrarily close to $1/2$. Hence, $m > \frac{25}{16}n$ for all sufficiently large $n$. 
Therefore,
\[ M(\mathcal{S}_E,\P) \simeq |\P|\frac{\diam{\mathcal{S}_{E'}}}{\diam{\mathcal{S}_E}}\geq (4p)^{\sqrt{m}-\sqrt{n}-1} \gtrsim (4p)^{\sqrt{m}-\sqrt{n}} \geq (4p)^{\sqrt{n}/4} \]
which goes to infinity as $n$ goes to infinity. Thus $\mathcal{S}$ is not a weak chord-arc curve.

\bibliographystyle{abbrv}
\bibliography{thesisref}

\end{document}